\documentclass[12pt,oneside,reqno]{amsart}
\usepackage{enumerate, mathtools,amssymb,mathrsfs}
\usepackage{url}
\usepackage{xspace}

\usepackage[margin=1in]{geometry}

\newcommand{\Z}{\mathbb{Z}}

\newcommand{\N}{\mathbb{N}}

\DeclareMathOperator{\pnt}{\raise 0.5mm \hbox{\large\textbf{.}}}

\DeclareMathSymbol{*}{\mathbin}{symbols}{"03}
\newcommand{\note}[2][ ]{}

\makeatletter
\newtheorem*{rep@theorem}{\rep@title}
\newcommand{\newreptheorem}[2]{%
\newenvironment{rep#1}[1]{%
 \def\rep@title{#2 \ref{##1}}%
 \begin{rep@theorem}}%
 {\end{rep@theorem}}}
\makeatother

\newtheorem{theorem}{Theorem}[section]

\newreptheorem{corollary}{Corollary to Conjecture}
\newtheorem{conjecture}[theorem]{Conjecture}
\newreptheorem{conjecture}{Conjecture}

\theoremstyle{definition}
\newtheorem{remark}[theorem]{Remark}

\DeclarePairedDelimiter{\ceil}{\lceil}{\rceil}

\usepackage[pdfauthor={Samuel D. Judge and Fabrizio Zanello},
            pdftitle={On the density of the odd values of the partition function},
            pdfsubject={partition theory},
            pdfkeywords={partitions, modular forms, congruences, modulo 2},
            pdfproducer={Latex with hyperref},
            pdfcreator={latex->dvips->ps2pdf},
            pdfpagemode=UseNone,
            bookmarksopen=false,
            bookmarksnumbered=true]{hyperref}

\begin{document}
\title[On the density of the odd values of the partition function, II]{On the density of the odd values of the partition function, II: An infinite conjectural framework}

\author[Samuel D. Judge]{Samuel D. Judge}
\address{Department of Mathematical  Sciences, Michigan Tech, Houghton, MI  49931}
\email{sdjudge@mtu.edu}

\author{Fabrizio Zanello}
\address{Department of Mathematical  Sciences, Michigan Tech, Houghton, MI  49931}
\email{zanello@mit.edu; zanello@mtu.edu}


\maketitle

\begin{abstract} We continue our study of a basic but seemingly intractable problem in integer partition theory, namely the conjecture that $p(n)$ is odd exactly $50\%$ of the time. Here, we greatly extend on our previous paper by providing a doubly-indexed, infinite framework of conjectural identities modulo 2, and show how to, in principle, prove each such identity. However, our conjecture remains open in full generality.

A striking consequence is that, under suitable existence conditions, if any $t$-multipartition function is odd with positive density and $t\not \equiv 0\ (\bmod\ 3)$, then $p(n)$ is also odd with positive density. These are all facts that appear virtually impossible to show unconditionally today.

Our arguments employ a combination of algebraic and analytic methods, including certain technical tools recently developed by Radu in his study of the parity of the Fourier coefficients of modular forms.

\end{abstract}

\section{Introduction}

Let $p(n)$ denote the number of \emph{partitions} of a nonnegative integer $n$, where a partition of $n$ is an unordered set of positive integers that sum to $n$. It is well known that the generating function of $p(n)$ is
\begin{equation}
\sum_{n=0}^{\infty} p(n) q^n =\frac{1}{\prod_{i=1}^{\infty} (1-q^i)}.\nonumber
\end{equation}
For any positive integer $t$, define $p_t(n)$ by 
\begin{equation}
\sum_{n=0}^{\infty} p_t(n)q^n = \frac{1}{\prod_{i=1}^{\infty} (1-q^i)^t}.\nonumber
\end{equation}
Hence, $p_t(n)$ denotes the number of $t$-\emph{multipartitions} (also called $t$-\emph{colored partitions}) of $n$. In particular, $p_1(n)=p(n)$. 

The goal of this paper is to provide additional insight into the long-standing question of estimating the density of the odd values of the partition function. As Paul Monsky once pointedly stated \cite{Monsky}, ``the best minds of our generation haven't gotten anywhere with understanding the parity of $p(n)$.'' It is a famous conjecture of Parkin and Shanks \cite{PaSh} that the even and the odd values of $p(n)$ are equidistributed (see also Section 4 of \cite{NiSa}). Currently, the best result, which has improved on the work of multiple authors (see for example \cite{Ahl,BelNic,Ei,Nic,Ono3}), is due to Bella\"iche, Green, and Soundararajan \cite{BGS} and states that the number of odd values of $p(n)$ for $n\leq x$ has at least the order of $\frac{\sqrt{x}}{\log\log x}$, for $x\to \infty$. In fact, their bound holds for any $t$-multipartition function $p_t(n)$, thus improving on a result of the second author (cf. \cite{Zan}, which includes an elementary proof for $t=3$). We note that the current record lower bound on the number of \emph{even} values of $p(n)$ has the order of $\sqrt{x}\log\log x$  (\cite{BelNic}. Unlike in the odd case, a lower bound of $\sqrt{x}$ for the even values is trivial).

We define the \emph{density} of the odd values of $p_t(n)$ in the natural way: 
\begin{equation}
\delta_t=\lim_{x\to \infty}\frac{\#\{n\leq x: p_t(n)\ \text{is odd}\}}{x},\nonumber
\end{equation} 
\noindent if the limit exists. Note that, for obvious parity reasons, it suffices to restrict our attention to the case of $t$ odd. Given the above bounds, we are still very far from showing that $\delta_t>0$, for any $t$. In fact, it is not even known at this time that any $\delta_t$ exists. 

In \cite{JKZ}, in collaboration with Keith, we generalized the conjecture of Parkin and Shanks (i.e., $\delta_1=1/2$) to every $t\geq 1$. Namely:
\begin{conjecture}[\cite{JKZ}, Conjecture 1]\label{JKZconj} The density $\delta_t$ exists and equals $1/2$, for any positive odd integer $t$. Equivalently, if $t = 2^k t_0$ with $t_0\ge 1$ odd, then $\delta_t$ exists and equals $2^{-k-1}$.
\end{conjecture}


Much work has been done on partition congruences modulo odd prime numbers $p$ (several papers include \cite{AhlOno, Mahl, Ono1, Radu1, Rama}), though most successfully for $p> 3$, i.e., for $p$ coprime to 24. For example, \cite{AndrGar, Atkin, AtkinSwinner,BerndtOno, Dyson, GKD, Rama2} proved congruences for the partition function when $p=5, 7, 11$ and powers of these primes, while in \cite{Atkin2, AtkinO, AtkinSwinner2}, analogous results were shown for small powers of certain other primes. Some of the most powerful results currently available are due to Ahlgren and Ono \cite{Ahl2,AhlOno,Ono1}, who proved that for any integer $m$ coprime to $24$, there exist infinitely many, nonnested arithmetic progressions $An+B$ such that $p(An+B)\equiv 0\ \pmod{m}$. Mahlburg \cite{Mahl} later explained combinatorially all such congruences by means of a celebrated partition statistic called crank. 

In this paper we focus on the case $p=2$, and greatly extend on our first work \cite{JKZ} by conjecturally providing a \emph{doubly-indexed, infinite} family of identities modulo $2$ which relate the various multipartition functions. We will show explicitly how to, in principle, prove \emph{any given} case of our conjecture using the analytic machinery of modular forms. However, we do not have a full proof for the conjecture nor for any infinite subfamily. 

A striking consequence, which again greatly generalizes \cite{JKZ}, is that (assuming existence of the $\delta_i$): $\delta_1>0$ if there is any $\delta_t>0$ with $t\not \equiv 0\ (\bmod\ 3)$; and $\delta_1+\delta_3>0$ if there is any $\delta_t>0$. As  noted earlier, an unconditional proof of any of these positivity results appears way beyond reach with the current technology. 

Finally, we will provide a completely algebraic argument for one of our congruences (see \cite{JKZ} for more such proofs). It remains an interesting open problem to find algebraic proofs for most other congruences.

\section{Formulation of the Main Conjectures}

For sake of completeness, we recall the main results from \cite{JKZ} without proof. 

\begin{theorem}[\cite{JKZ}, Theorem 2]\label{JKZ1} Assume all densities $\delta_i$ below exist. Then $\delta_t>0$ implies $\delta_1>0$ for
\begin{equation}
t=5,7,11,13,17,19,23,25. \nonumber
\end{equation}
Moreover, $\delta_t > 0$ implies $\delta_r > 0$ for the following pairs $(t,r)$:
\begin{equation}
(27,9), (9,3), (25,5), (15,3), (21,3), (27,3).\nonumber
\end{equation}
\end{theorem}

Remember that two power series, $\sum_n a(n)q^n$ and $\sum_n b(n)q^n$, are \emph{congruent modulo} $m$ if $a(n)\equiv b(n)$ (mod $m$) for all $n$. Unless specified otherwise, all congruences from here on will be modulo $2$. An essential tool to prove Theorem \ref{JKZ1} was given by the following theorem, which we state in a slightly different form.

\begin{theorem} [\cite{JKZ}, Theorem 3] \label{JKZ2}
The congruence 
\begin{equation} \label{eq:conj}
q\sum_{n=0}^{\infty}p_t(a n +b)q^n\equiv \frac{1}{\prod_{i=1}^{\infty}(1-q^i)^{at}}+\frac{1}{\prod_{i=1}^{\infty}(1-q^{ai})^t}
\end{equation}
holds for the following twelve pairs $(a,t)$:
\begin{equation}
(5,1), (7,1), (11,1), (13,1), (17,1), (19,1), (23,1), (3,3), (5,3), (7,3), (5,5), (3,9),\nonumber
\end{equation}
where \[
    b= 
\begin{cases}
    \frac{t}{3}\cdot8^{-1}\ (\bmod\ a),& \text{if } 3\vert t;\\
    t\cdot24^{-1}\ (\bmod\ a),              & \text{otherwise}.
\end{cases}
\]
\end{theorem}

In this work, we frame the above congruences in a much broader context. Specifically, we conjecture that for any odd integers $a$ and $t$ such that $3\vert t$ if $3\vert a$, an identity similar to (though usually more complicated than)  (\ref{eq:conj}) exists. From this, we derive an \emph{infinite class} of ``delta implications'' (under the assumption all relevant $\delta_i$ exist) analogous to those given in Theorem \ref{JKZ1}. The following lays this out explicitly for $t=1$. As usual, we denote by $\lfloor \alpha \rfloor$ (respectively, $\lceil \alpha \rceil$) the nearest integer to $\alpha$ that is less (respectively, greater) than or equal to $\alpha$.

\begin{conjecture}\label{mainconj} 
Fix a positive integer $a$ coprime to $24$. Let $b=24^{-1}\ (\bmod\ a)$ (or in the trivial case $a=1$, let $b=0$) and $k=\ceil[\big]{\frac{a^2-1}{24a}}$. Then
\begin{equation} \label{CongOfForm}
q^k \sum_{n=0}^{\infty} p(an+b)q^n \equiv \sum_{d|a} \sum_{j=0}^{\lfloor k/d \rfloor} \frac{\epsilon^1_{a,d,j}\ q^{dj}}{\prod_{i\geq 1} (1-q^{di})^{\frac{a}{d}-24j}},
\end{equation}
for a suitable choice of the $\epsilon^1_{a,d,j}\in\{0,1\}$, where $\epsilon^1_{a,1,0}=1$ and $\epsilon^1_{a,d,j}=0$ if $a/d-24j<0$. 
\end{conjecture}

\begin{repcorollary}{mainconj} 
Fix any positive integer $a\not \equiv 0\ (\bmod\ 3)$, and suppose all necessary $\delta_i$ exist. If $\delta_a>0$, then $\delta_1>0$.  
\end{repcorollary} 

Something of note is what we mean by ``necessary $\delta_i$'' in the above corollary. Focusing without loss of generality on $a$ odd, it is certainly sufficient to assume that for all $1\leq i\leq a$, $i\equiv \pm\ 1\ (\bmod\ 6)$, $\delta_i$ exists. However, in general, only a small fraction of such $\delta_i$ are actually needed in order to obtain a given delta implication, though determining explicitly which seems extremely hard.

The above conjecture is the main special case ($t=1$) of the following. It provides an infinite, two-dimensional set of conjectural partition identities modulo $2$.

\begin{conjecture}\label{biggerconj} 
Fix any positive odd integers $a$ and $t$, where $3\vert t$ if $3\vert a$. Let $k=\ceil[\big]{\frac{t(a^2-1)}{24a}}$. Then
\begin{equation}\label{CongOfForm2}
q^k \sum_{n=0}^{\infty} p_t(an+b)q^n \equiv \sum_{d|a} \sum_{j=0}^{\lfloor k/d \rfloor} \frac{\epsilon^t_{a,d,j}\ q^{dj}}{\prod_{i\geq 1} (1-q^{di})^{\frac{at}{d}-24j}},
\end{equation}
where \[
    b= 
\begin{cases}
   \ \ \ 0, &\text{if }\ a=1;\\
    \frac{t}{3}\cdot8^{-1}\ (\bmod\ a),& \text{if }\ 3\vert t;\\
    t\cdot24^{-1}\ (\bmod\ a),              & \text{otherwise},
\end{cases}
\]
for a suitable choice of the $\epsilon^t_{a,d,j}\in\{0,1\}$, where $\epsilon^t_{a,1,0}=1$ and $\epsilon^t_{a,d,j}=0$ if $at/d-24j<0$.
\end{conjecture}

\begin{repcorollary}{biggerconj} 
Fix any positive integer $A$ and suppose all necessary $\delta_i$ exist. If $\delta_{A}>0$, then $\delta_1+\delta_3>0$. 
\end{repcorollary} 

\begin{remark}
Perhaps unsurprisingly, no partition congruence similar to the above and implying $\delta_3>0 \implies \delta_1>0$ appears to exist. In fact, we have not been able to determine \emph{any} congruence, even on a conjectural level, that yields such a basic delta implication.
\end{remark}

\section{Modular Form Preliminaries}

We now present the necessary facts and notation coming from modular forms, along with the main theorem of Radu \cite{Radu}; for basic facts, proofs, and further details, we refer the reader to \cite{Koblitz,Ono}. The notation used below mainly follows \cite{Radu}.

Fix $N\geq 1$ and let $\Z^{d(N)}$ denote the set of integer tuples with entries $r_\delta$ indexed by the positive divisors $\delta$ of $N$.  For a given $r=(r_{\delta}) \in \Z^{d(N)}$, define
\begin{equation}
\omega(r) = \sum_{\delta \vert N} r_\delta \, ; \quad \sigma_\infty (r) = \sum_{\delta \vert N} \delta r_\delta \, ; \quad \sigma_0(r) = \sum_{\delta \vert N} \frac{N}{\delta} r_\delta \, ; \quad \Pi(r) = \prod_{\delta \vert N} \delta^{\vert r_\delta \vert}.\nonumber
\end{equation}
Additionally, let $\sum_{n\geq 0} \alpha_r(n) q^n = \prod_{\delta \vert N} \prod_{i\geq1} (1-q^{\delta i})^{r_\delta}$, and define
\begin{equation}
g_{a,b} (q) = q^{\frac{24b+\sigma_\infty (r)}{24a}} \sum_{n=0}^\infty \alpha_r(an+b)q^n.\nonumber
\end{equation}

Now let $\Delta^*$ be the set of tuples $(a,N,M,b,(r_\delta)) \in (\mathbb{N^{+}})^3  \times \{0,1,\dots,a-1\} \times \Z^{d(N)}$ such that the following hold: For all primes $p$, if $p\vert a$ then $p\vert M$; if $r_\delta \neq 0$, then $\delta \vert aM$; and if $\kappa = \gcd(a^2-1,24)$, then
\begin{equation}
24 {\ }\vert {\ } \kappa\frac{aM^2}{N} \sigma_0(r); \quad 8 {\ }\vert {\ } \kappa M \omega(r); \quad \frac{24a}{\gcd(\kappa (24b+\sigma_\infty(r)),24a)} {\ }\vert {\ } M.\nonumber
\end{equation}

We define $\Gamma_0(N)$ as the subgroup of $SL_2(\mathbb{Z})$ consisting of the matrices $\left( \begin{matrix} v & w \\ x & y \end{matrix} \right)$ where $ x \equiv 0 \pmod{N}$. For $a,N \in \mathbb{N}^{+}$, $b\in \{0,1,\dots,a-1\}$, and $r\in \Z^{d(N)}$, let $P_{a,r}(b)$ be the set of residues modulo $a$ that can be written as $b\omega^2+(\omega^2-1)\sigma_{\infty}(r)/24$ for some integer $\omega$ with $\gcd(\omega,24)=1$. Further, let
\begin{equation}
\chi_{a,r}(b) = \prod_{\ell \in P_{a,r}(b)} e^{2\pi i\frac{(1-a^2)(24\ell+\sigma_{\infty}(r))}{24a}}.\nonumber
\end{equation}

Set $f$ to be a modular form of weight $k$ and character $\chi$ for $\Gamma_0(N)$. We let the \emph{level} of $f$ be the least value of $N$ for which $f$ is a modular form of weight $k$ for $\Gamma_0(N)$. An \emph{$\eta$-quotient} is a quotient of powers of the Dedekind $\eta$-function, $\eta(\tau)$, and magnifications thereof, $\eta(C\tau)$, where
\begin{equation}
\eta(\tau) = q^{\frac{1}{24}} \prod_{i=1}^\infty (1-q^i),\ \ \ q = e^{2\pi i \tau}.\nonumber
\end{equation}

The following theorem of Gordon, Hughes, and Newman gives conditions sufficient for any given $\eta$-quotient to be a modular form.
\begin{theorem} [\cite{GH,Newman}]\label{ghn} Let $N\geq 1$ be an integer and $f(\tau) = \prod_{\delta \vert N} \eta^{r_\delta} (\delta \tau)$, with $r=(r_{\delta})\in \Z^{d(N)}$.  If
\begin{equation}
\sigma_{\infty}(r)\equiv \sigma_0(r)\equiv 0\ (\bmod\ 24)\nonumber,
\end{equation}
then $f$ is a modular form of weight $k=\frac{1}{2} \sum r_\delta$, level dividing $N$, and character $\chi(d) = \left( \frac{(-1)^k\cdot \beta}{d} \right)$, where $\beta = \prod_{\delta \vert N} \delta^{r_\delta}$ and $\left(\frac{\cdot}{d}\right)$ denotes the Jacobi symbol.
\end{theorem}

Radu's main result can be phrased as follows:
\begin{theorem} [\cite{Radu}]\label{RaduThm}
Let $(a,N,M,b,r=(r_\delta)) \in \Delta^*$, $s=(s_\delta) \in \Z^{d(M)}$, and let $\nu$ be an integer such that $\chi_{a,r}(b) = e^{\frac{2\pi i\nu}{24}}$ (such an integer is guaranteed to exist by \cite{Radu}, Lemma 43). Define
\begin{equation}
F(s,r,a,b) (\tau) = \prod_{\delta \vert M} \eta^{s_\delta} (\delta \tau) \prod_{u \in P_{a,r}(b)} g_{a,u}(q).\nonumber
\end{equation}
Then $F$ is a weakly holomorphic modular form of weight zero and trivial character for $\Gamma_0(M)$ if and only if the following conditions simultaneously hold:
\begin{align}
\vert P_{a,r} (b) \vert \cdot \omega(r) + \omega(s) = 0; \\
 \nu + \vert P_{a,r} (b) \vert \cdot a \sigma_\infty (r) + \sigma_\infty (s) \equiv 0 \pmod{24}; \\
\vert P_{a,r} (b) \vert \cdot \frac{aM}{N}\sigma_0(r) + \sigma_0(s) \equiv 0 \pmod{24}; \\
\Pi(s)\cdot \prod_{\delta \vert N} (a \delta)^{\vert r_\delta \vert\cdot {\vert P_{a,r} (b) \vert } } \, \text{ is a perfect square.}
\end{align}
\end{theorem}
\noindent (Note that, in the above statement, $|r_{\delta}|$ refers to an absolute value while $|P_{a,r}(b)|$ indicates the cardinality of a set.) We now present results of Ligozat and Radu that bound the order of the cusps for an $\eta$-quotient and the modular function $F(s,r,a,b)$, respectively. Our statements slightly differ from the original papers. 

\begin{theorem}[\cite{Koblitz,Ligo,Ono}]\label{ligozat} Let $N$ be a positive integer. If $f(\tau)=\prod_{\delta \vert N} \eta^{r_\delta} (\delta \tau)$ is an $\eta$-quotient satisfying the conditions of Theorem \ref{ghn} for $\Gamma_0(N)$, then $f$ can only have cusps at rational numbers of the form $\frac{\gamma}{\mu}$, where $\mu\vert N$ and $\gcd(\gamma,\mu)=1$. When $f$ has a cusp $\frac{\gamma}{\mu}$, its order is the absolute value of
\begin{equation}
\frac{N}{24} \sum_{\delta \vert N} \frac{\left(\gcd(\mu,\delta)\right)^2\cdot r_\delta}{\gcd \left(\mu,\frac{N}{\mu}\right)\cdot \mu \delta}.
\end{equation}
\end{theorem}

\begin{theorem}[\cite{Radu}, Theorem 47 and Equations (56-57)]\label{RaduThm47}
For $F(s,r,a,b)$ and the corresponding $r\in \Z^{d(N)}$ and $s\in \Z^{d(M)}$, as constructed above and satisfying the conditions of Theorem \ref{RaduThm}, the order of $F$ at any cusp is uniformly bounded from above by the absolute value of 
\begin{equation} {{\emph{min}} \atop {n \in \N}} \frac{M}{\gcd(n^2,M)} \left( \vert P_{a,r}(b) \vert {{\emph{min}} \atop {{m \vert a} \atop {\gcd(m,n)=1}}} \frac{1}{24} \sum_{\delta \vert N} r_\delta \frac{(\gcd(\delta m,an))^2}{\delta a} + \frac{1}{24} \sum_{\delta \vert M} s_\delta \frac{(\gcd(\delta,n))^2}{\delta} \right).\nonumber
\end{equation}
\end{theorem}

Finally, a classical result of Sturm \cite{sturm} gives sufficient conditions to equate two modular forms modulo a given prime. We phrase it in the following way.
\begin{theorem} [\cite{sturm}]\label{sturmthe} Let $p$ be a prime number, and $f(\tau) = \sum_{n=n_0}^\infty a(n) q^n$ and $g(\tau) = \sum_{n=n_1}^\infty b(n) q^n$  be holomorphic modular forms of weight $k$ for $\Gamma_0(N)$ of characters $\chi_1$ and $\chi_2$, respectively, where $n_0,n_1\in \mathbb{N}$.  If either $\chi_1=\chi_2$ and
\begin{equation}
a(n) \equiv b(n) \pmod{p} \quad \text{for all} \quad n\le \frac{kN}{12}\cdot \prod_{d{\ }\emph{prime};{\ } d \vert N} \left(1+\frac{1}{d^2}\right),\nonumber
\end{equation}
or $\chi_1\neq\chi_2$ and
\begin{equation}
a(n) \equiv b(n) \pmod{p} \quad \text{for all} \quad n\le \frac{kN^2}{12}\cdot \prod_{d{\ }\emph{prime};{\ } d \vert N} \left(1-\frac{1}{d^2}\right),\nonumber
\end{equation}
then $f(\tau) \equiv g(\tau) \pmod{p}$ (i.e., $a(n) \equiv b(n) \pmod{p}$ for all $n$).
\end{theorem}

\section{Proofs}

We begin with a proof for Corollary to Conjecture \ref{mainconj}. The argument for Corollary to Conjecture \ref{biggerconj} follows by essentially the same reasoning, and therefore will be omitted (see Remark \ref{importantRemark}). We restate the corollary here for reference: 

\begin{repcorollary}{mainconj} 
Fix any positive integer $a\not \equiv 0\ (\bmod\ 3)$, and suppose all necessary $\delta_i$ exist. If $\delta_a>0$, then $\delta_1>0$. 
\end{repcorollary} 

\begin{proof} We proceed by induction on $a$, which we can assume to be odd for evident parity reasons. The base case $a=1$ is clear. We first recall from \cite{JKZ} a proof of the case $a=5$, as this will give insight into the logic we will use for the remainder of the argument. We use the corresponding identity given in (\ref{eq:conj}) (which was proved in \cite{JKZ}), namely,
\begin{equation}
q\sum_{n=0}^{\infty} p(5n+4)q^n \equiv \frac{1}{\prod_{i\geq 1} (1-q^i)^5} +\frac{1}{\prod_{i\geq 1} (1-q^{5i})}.\nonumber
\end{equation}

Suppose that $\delta_5>0$ and $\delta_1=0$. Therefore, $\#\{n\leq x: p_5(n)\ \text{is odd}\}=\delta_5 x+o(x)$, while the number of odd coefficients up to $x$ of $1/\prod_{i\geq 1} (1-q^{5i})=\sum_{n=0}^{\infty} p(n)q^{5n}$ is $o(x)$. Hence, the number of odd coefficients up to $x$ of $\sum_{n=0}^{\infty} p(5n+4)q^n$ is also $\delta_5 x+ o(x)$, which yields $\delta_1\geq \delta_5/5>0$, a contradiction. This shows the case $a=5$.

Now suppose the result holds for all cases up to and including an arbitrary $a-4\geq 5$, and assume $\delta_1=0$. The next case is either $a-2$ or $a$. We assume it is $a$ and, consequently, that $\delta_a>0$; the proof for the case $a-2$ is entirely similar. By Conjecture \ref{mainconj}, there exists some identity of the form
\begin{equation}
q^k \sum_{n=0}^{\infty} p(an+b)q^n \equiv \sum_{d|a} \sum_{j=0}^{\lfloor k/d \rfloor} \frac{\epsilon^1_{a,d,j}\ q^{dj}}{\prod_{i\geq 1} (1-q^{di})^{\frac{a}{d}-24j}},\nonumber
\end{equation}
where $\epsilon^1_{a,1,0}=1$. By assumption, the number of odd coefficients of $q^k \sum_{n=0}^{\infty} p(an+b)q^n$ up to $x$ must be $o(x)$, and the number of odd coefficients of $1/\prod_{i\geq 1} (1-q^i)^{a}$ is $\delta_{a} x+o(x)$. 

Therefore, at least one additional term in the finite double sum on the right side must give positive density. Suffices from here to note that an additional term giving positive density must be of the form 
\begin{equation}
\frac{q^{d_0j_0}}{\prod_{i\geq 1} (1-q^{d_0i})^{\frac{a}{d_0}-24j_0}},\nonumber
\end{equation}
for suitable $d_0$ and $j_0$ with $(d_0,j_0)\neq (1,0)$, and that $a/d_0-24j_0$ is always $\pm\ 1\ (\bmod\ 6)$. By the inductive assumption, since $\delta_{\frac{a}{d_0}-24j_0}>0$, then $\delta_1>0$. This contradiction gives the result. 
\end{proof}

\begin{remark}\label{importantRemark}
The proof of Corollary to Conjecture \ref{biggerconj} is entirely similar. Indeed, we can again consider $A$ odd and then set $A=at$, with $a$ and $t$ as in the conjecture. Now the same reasoning yields that, if $\delta_A>0$ and all the relevant $\delta_i$ exist, at least one between $\delta_1$ and $\delta_3$ must be positive. 
\end{remark}

We next show, as a sample, one special case of Conjecture \ref{mainconj}, namely when $a=31$ and $t=1$. (Consequently this proves that $\delta_{31}>0\implies \delta_1>0$, always under the assumption the necessary $\delta_{i}$ exist.) In principle, \emph{any given identity} of Conjectures \ref{mainconj} and \ref{biggerconj} can be shown using the exact same strategy, as we will explicitly outline in the proof. 

We will employ a technique introduced by Radu \cite{Radu}, which was first used in \cite{JKZ}. Our approach can be summarized as follows: We start with the identity for $q^k\sum_{n=0}^\infty p_t(an+b)q^n$ given by Conjecture \ref{biggerconj}, and consider the subset of $\Delta^*$ where $N=1$ and $r=(r_1)=(-t)$. Then we determine an integer $M\geq 1$ and an $s$-vector $s$ such that the corresponding function $F$ satisfies the conditions of Theorem \ref{RaduThm}. If in the process we encounter modular forms that are weakly holomorphic, we clear all poles by multiplying by another modular form of sufficiently high order (specifically, some $\eta(4\tau)^{24i}$ will always work). This will be done according to Theorems \ref{ligozat} and \ref{RaduThm47}. As a final step, we use Sturm's bound (Theorem \ref{sturmthe}) to complete the proof. 

Note that as $a$ increases, a much higher power of $\eta(4\tau)^{24}$ is needed to clear out the poles, thus making the computational cost necessary to verify Sturm's bound significantly greater. Further, there appears to be no infinite family of congruences coming from either Conjecture \ref{mainconj} or \ref{biggerconj} for which one can apply a \emph{uniform} Sturm bound, nor any such family with a bounded number of terms. In light of the above reasons, we have not yet been able to prove our conjectures in any form beyond a case-by-case basis.

\begin{theorem}\label{identities1} 
\begin{equation}\label{31andStuff}
q^2\sum_{n=0}^{\infty} p(31n+22)q^n \equiv \frac{1}{\prod_{i\geq 1} (1-q^i)^{31}}+\frac{q}{\prod_{i\geq 1} (1-q^i)^{7}}+\frac{1}{\prod_{i\geq 1} (1-q^{31i})}.
\end{equation}
\end{theorem}

\begin{proof} 
We will provide as broad a proof as possible for arbitrary $a$ and $t$ as in Conjecture \ref{biggerconj}, and only look at the specific case of the statement ($a=31$ and $t=1$) in the second part of the argument.  We begin by using the function $\alpha_r$ defined at the beginning of Section 3, by setting
\begin{equation}
\sum_{n\geq 0} \alpha_r(n)q^n=\sum_{n\geq 0} p_t(n)q^n=\frac{1}{\prod_{i\geq 1} (1-q^i)^t}.
\end{equation}
Therefore, $N=1$ and $r=(r_1)=(-t)$ . Now we consider $P_{a,r}(b)$. Notice that $\sigma_{\infty} (r)=-t,$ and so 
\begin{equation}
b\omega^2+\frac{\omega^2-1}{24}\sigma_{\infty}(r) =b\omega^2+\frac{t(1-\omega^2)}{24}.\nonumber
\end{equation}

Note that $24b\equiv t\ (\bmod\ a)$, and as such, the above becomes 
\begin{equation}
b\omega^2+\frac{24b(1-\omega^2)}{24}= b\omega^2+b-b\omega^2= b \ (\bmod\ a).\nonumber
\end{equation} 
Hence $P_{a,r}(b)=\{b\}$, for any $a$, $b$, and $t$ satisfying the conditions in Conjecture \ref{biggerconj}. We now notice that
\begin{equation}
g_{a,b}(q)=q^{\frac{24b-t}{24a}}\sum_{n\geq 0} p_t(an+b)q^n;\nonumber
\end{equation} 
\begin{equation}
\chi_{a,r}(b)=e^{\frac{2\pi i(24b-t)(1-a^2)}{24a}};\nonumber
\end{equation}
\begin{equation}
\nu=\frac{(1-a^2)(24b-t)}{a}.
\end{equation}

Finally, since $|P_{a,r}(b)|=1$, by Theorem \ref{RaduThm} we have
\begin{equation}
F(s,r,a,b)(\tau)=\prod_{\delta\vert M}\eta^{s_{\delta}}(\delta\tau)\ \ q^{\frac{24b-t}{24a}}\sum_{n\geq 0} p_t(an+b)q^n,
\end{equation}
for some $M$ and some $s$-vector $s=(s_{\delta})$. 

Consider now the specific case of the statement, where $a=31$, $t=1$, and $b=22$. We have $N=1$, and $r=(r_1)=(-1)$. Additionally, we choose $M=2\cdot31=62$. As proved above, $P_{31,r}(22)=\{22\}$. Moving to the conditions of Theorem \ref{RaduThm}, standard computations imply that we are looking for an $s$-vector $s$ such that the following simultaneously hold: 
\begin{equation}
\omega(s)=1;\nonumber
\end{equation}
\begin{equation}
\sigma_{\infty}(s)\equiv 7\ (\bmod\ 24);\nonumber
\end{equation}
\begin{equation}
\sigma_{0}(s)\equiv 2\ (\bmod\ 24);\nonumber
\end{equation}
\begin{equation}
31\cdot\Pi(s)\ \text{is a perfect square.}\nonumber
\end{equation}

It is not hard to check that  $s=(s_1,s_2,s_{31},s_{62})=(3,1,4,-7)$ satisfies these conditions. This $s$-vector results in 
\begin{equation}\label{LHS}
F(s,r,31,22)(\tau)=q^{\frac{17}{24}} \frac{\eta(\tau)^3\eta(2\tau)\eta(31\tau)^{4}}{\eta(62\tau)^{7}}\sum_{n=0}^{\infty} p(31n+22)q^n.
\end{equation}
By Theorem \ref{RaduThm}, this is a modular form of weight zero. Thus, in order to obtain $F(s,r,31,22)$ on the left side of Equation (\ref{31andStuff}), we have to multiply this latter by 
\begin{equation}\label{sVector}
q^{\frac{-31}{24}}\frac{\eta(\tau)^3\eta(2\tau)\eta(31\tau)^{4}}{\eta(62\tau)^{7}}.
\end{equation}
\indent Now consider the right side of (\ref{31andStuff}). Expressing it in terms of $\eta$-quotients, it becomes:  
\begin{equation}\label{Equ10}
\frac{q^{\frac{31}{24}}}{\eta(\tau)^{31}}+\frac{q^{\frac{31}{24}}}{\eta(\tau)^7}+\frac{q^{\frac{31}{24}}}{\eta(31\tau)}.
\end{equation}

By multiplying Formulas (\ref{sVector}) and (\ref{Equ10}) together, we obtain
\begin{equation}\label{Equ11}
\frac{\eta(2\tau)\eta(31\tau)^{4}}{\eta(\tau)^{28}\eta(62\tau)^{7}}+\frac{\eta(2\tau)\eta(31\tau)^{4}}{\eta(\tau)^{4}\eta(62\tau)^{7}}+\frac{\eta(\tau)^3\eta(2\tau)\eta(31\tau)^{3}}{\eta(62\tau)^{7}}.
\end{equation} 

At this point, in order to match the weight of the modular form $F$, we want to turn (\ref{Equ11}) into a weight zero modular form as well. Using the fact that $\eta(2\tau)\equiv \eta(\tau)^2$, we can rewrite (\ref{Equ11}) as 
\begin{equation}\label{RHS}
\frac{\eta(\tau)^{32}\eta(31\tau)^{4}}{\eta(2\tau)^{29}\eta(62\tau)^{7}}+\frac{\eta(\tau)^{8}\eta(31\tau)^{4}}{\eta(2\tau)^5\eta(62\tau)^{7}}+\frac{\eta(\tau)^3\eta(2\tau)\eta(31\tau)^{3}}{\eta(62\tau)^{7}}.
\end{equation}

This is a modular form of weight zero, by Theorem \ref{ghn} along with the fact that the sum of weight zero modular forms is a weight zero modular form. It will be sufficient to show that the right sides of (\ref{LHS}) and (\ref{RHS}) are identical modulo $2$. Before using Sturm's Theorem, we must first clear all possible cusps. To do so, notice that the bound given in Theorem \ref{RaduThm47}, uniform on all cusps of $F(s,r,31,22)$, is
\begin{equation}\left | {{\min} \atop {n \in \N}} \frac{62}{\gcd(n^2,62)} \left( \vert P_{31,r}(22) \vert {{\min} \atop {{m \vert 31} \atop {\gcd(m,n)=1}}} \frac{1}{24} \sum_{\delta \vert 1} r_\delta \frac{(\gcd(\delta m,31n))^2}{31\delta } + \frac{1}{24} \sum_{\delta \vert 62} s_\delta \frac{(\gcd(\delta,n))^2}{\delta} \right)\right |\nonumber
\end{equation}
\begin{equation}
=\frac{567}{8}=70.875.\nonumber
\end{equation}

Theorem \ref{ligozat} allows us to make a similar calculations to clear the possible cusps of (\ref{RHS}), resulting in a bound on the order of 34. Thus, a power of $\eta(4\tau)$ sufficient to simultaneously clear all cusps of (\ref{LHS}) and (\ref{RHS}) is $24\cdot14= 336$. 

It follows that the identity to be checked is 
\begin{equation}
q^{\frac{17}{24}}\eta(4\tau)^{336} \frac{\eta(\tau)^3\eta(2\tau)\eta(31\tau)^{4}}{\eta(62\tau)^{7}}\sum_{n=0}^{\infty} p(31n+22)q^n\stackrel{?}{\equiv} \nonumber
\end{equation}
\begin{equation}
\eta(4\tau)^{336}\left[\frac{\eta(\tau)^{32}\eta(31\tau)^{4}}{\eta(2\tau)^{29}\eta(62\tau)^{7}}+\frac{\eta(\tau)^{8}\eta(31\tau)^{4}}{\eta(2\tau)^5\eta(62\tau)^{7}}+\frac{\eta(\tau)^3\eta(2\tau)\eta(31\tau)^{3}}{\eta(62\tau)^{7}}\right],\nonumber
\end{equation}
where both sides are holomorphic modular forms of weight $k=0+\frac{336}{2}=168$. 
A straightforward application of Theorem \ref{sturmthe} now gives a Sturm bound of at most $416,640$. Finally, a Mathematica calculation verifies that this congruence holds modulo $2$ up to that point, completing the proof. 
\end{proof}

We now present a completely algebraic proof for the case $a=49$ of Conjecture \ref{mainconj} (and therefore for the delta implication $\delta_{49}>0 \implies \delta_1>0$, under the assumption that all of $\delta_{25}$, $\delta_{7}$, and $\delta_1$ exist, since we proved algebraically in \cite{JKZ} that both $\delta_{25}>0$ and $\delta_7>0$ imply $\delta_1>0$). For sake of simplicity, we will use the notation 
\begin{equation}
f_k=f_k(q)=\prod_{i=1}^{\infty} (1-q^{ki}).\nonumber
\end{equation}
\begin{theorem}\label{identities2} 
\begin{equation}
q^3 \sum_{n=0}^{\infty} p(49n+47)q^n \equiv \frac{1}{f_1^{49}}+\frac{q}{f_1^{25}}+\frac{q^2}{f_1}+\frac{1}{f_7^7}.\nonumber
\end{equation}
\end{theorem}

\begin{proof}
We deduce, from Zuckerman's identity for $\sum_{n=0}^{\infty} p(49n+47)q^n$ (see \cite{Zuckerman}), the following congruence:
\begin{equation}\label{zuck}
\sum_{n=0}^{\infty} p(49n+47) q^n \equiv q^2\frac{f_7^{12}}{f_1^{13}}+q^6\frac{f_7^{28}}{f_1^{29}}+q^5\frac{f_7^{24}}{f_1^{25}}+q^{13}\frac{f_7^{56}}{f_1^{57}}.
\end{equation}
Since we are working modulo 2, observe that the right side of (\ref{zuck})  can be rewritten as
\begin{equation}\label{zuck2}
q^2f_1^3y^4+q^5f_1^7y^8,
\end{equation}
if we  set
\begin{equation}\label{y}
y=  \frac{f_7^3}{f_1^4}+q\frac{f_7^7}{f_1^8}.
\end{equation}

We now employ an algebraic result of Lin (\cite{Lin}, Equation 2.4), namely:
\begin{equation}\label{LIN}
f_1 f_7+f_1^8 \equiv q^2f_7^8+qf_1^4f_7^4.
\end{equation}
(Alternatively, (\ref{LIN}) can also be derived, after standard manipulations, from the case $(a,t)=(7,1)$ of (\ref{eq:conj}) along with Ramanujan's congruence for $\sum_{n=0}^{\infty} p(7n+5)q^n$; see \cite{JKZ} and \cite{BerndtOno}.) Dividing both sides of (\ref{LIN}) by $qf_1^8f_7$, (\ref{y}) easily becomes:
\begin{equation}\label{yy}
y\equiv \frac{q^{-1}}{f_1^7}+\frac{q^{-1}}{f_7}.
\end{equation}

Thus, substituting (\ref{yy}) into (\ref{zuck2}) gives:
$$\sum_{n=0}^{\infty} p(49n+47) q^n \equiv q^2f_1^3\left( \frac{q^{-1}}{f_1^7}+\frac{q^{-1}}{f_7} \right)^4+q^5f_1^7\left(\frac{q^{-1}}{f_1^7}+\frac{q^{-1}}{f_7} \right)^8.$$
Multiplication by $q^3$ on both sides and some simple algebra modulo 2 now yield:
\begin{equation}\label{almostThere}
q^3\sum_{n=0}^{\infty} p(49n+47) q^n \equiv \frac{1}{f_1^{49}}+\frac{q}{f_1^{25}}+\left(\frac{f_1^7}{f_7^8}+q\frac{f_1^3}{f_7^4}\right).
\end{equation}

Finally, we can divide both sides of (\ref{LIN}) by $f_1f_7^8$ and rearrange to obtain:
\begin{equation}\label{last}
\frac{f_1^7}{f_7^8}+q\frac{f_1^3}{f_7^4}\equiv \frac{q^2}{f_1}+\frac{1}{f_7^7}.
\end{equation}
Substituting (\ref{last}) into the right side of (\ref{almostThere}) completes the proof.
\end{proof}

\section{Acknowledgements} This work will be part of the first author's Ph.D. dissertation, written at Michigan Tech under the supervision of the second author. Both authors wish to thank William Keith for many helpful discussions. The final version of this paper was drafted during a visiting professorship of the second author at MIT in Fall 2017, for which he is grateful to Richard Stanley. He was also partially supported by a Simons Foundation grant (\#274577). Finally, the second author would like to acknowledge that the order is not only alphabetical: The first author has contributed for well over $50\%$ to the research of this paper.


\begin{thebibliography}{99}

\bibitem{Ahl} S. Ahlgren: \emph{Distribution of parity of the partition function in arithmetic progressions}, Indag. Math. (N.S.) \textbf{10} (1999), 173--181.

\bibitem{Ahl2} S. Ahlgren: \emph{Distribution of the Partition Function Modulo Composite Integers}, Math. Ann. \textbf{318} (2000), 795--803.

\bibitem{AhlOno} S. Ahlgren and K. Ono: \emph{Congruence properties for the partition function}, Proc. Natl. Acad. Sci. USA \textbf{98} (2001), no. 23, 12882--12884.

\bibitem{AndrGar} G. Andrews and G. Garvan: \emph{Dyson's Crank of a Partition}, Bull. Amer. Math. Soc. (N.S.) \textbf{18} (1988), 167--171.

\bibitem{Atkin} A.O.L. Atkin: \emph{Proof of a Conjecture of Ramanujan}, Glasgow Math. J. \textbf{8} (1967), 14--32. 

\bibitem{Atkin2} A.O.L. Atkin: \emph{Multiplicative Congruence Properties and Density Problems for $p(n)$}, Proc. London Math. Soc. \textbf{18} (1968), 563--576. 

\bibitem{AtkinO} A.O.L. Atkin and J.N. O'Brien: \emph{Some Properties of $p(n)$ and $c(n)$ Modulo Powers of $13$}, Trans. Amer. Math. Soc. \textbf{126} (1967), 442--459.

\bibitem{AtkinSwinner} A.O.L Atkin and H.P.F. Swinnerton-Dyer: \emph{Some Properties of Partitions}, Proc. London Math. Soc. \textbf{4} (1954), 84--106.

\bibitem{AtkinSwinner2} A.O.L Atkin and H.P.F. Swinnerton-Dyer: \emph{Modular Forms on Noncongruence Subgroups}, Combinatorics (Proc. Sympos. Pure Math., Vol XIX, Univ. California, Los Angeles, Calif., (1968)) \textbf{19} (1971), 1--25.

\bibitem{BGS} J. Bella\"iche, B. Green, and K. Soundararajan: \emph{Non-zero Coefficients of Half-Integral Weight Modular Forms Mod $\ell$}, Res. Math. Sci. (Special Volume in honor of Don Zagier), to appear (arXiv:1704.07440). 

\bibitem{BelNic} J. Bella\"iche and J.-L. Nicolas: \emph{Parit\'e des coefficients de formes modulaires}, Ramanujan J. \textbf{40} (2016), no. 1, 1--44.

\bibitem{BerndtOno} B. Berndt and K. Ono:  \emph{Ramanujan's unpublished manuscript on the partition and tau functions with proofs and commentary},  Andrews Festschrift, D. Foata and  G.-N. Han Eds., Springer (2001), 39--110.

\bibitem{Dyson} F.J. Dyson: \emph{Some Guesses in the Theory of Partitions}, Eureka \textbf{8} (1944), 10--15.

\bibitem{Ei} D. Eichhorn: \emph {A new lower bound on the number of odd values of the ordinary partition function}, Ann. Comb. \textbf{13} (2009), 297--303.

\bibitem{GKD} F. Garvan, D. Kim, and D. Stanton: \emph{Cranks and $t$-cores}, Invent. Math. \textbf{101} (1990), no. 1, 1--17.

\bibitem{GH} B. Gordon and K. Hughes: \emph{Multiplicative properties of $\eta$-products, II},  Contemp. Math. \textbf{143} (1993), 415--430.

\bibitem{JKZ} S. Judge, W. Keith, and F. Zanello: \emph{On the Density of the Odd Values of the Partition Function}, Ann. Comb., to appear (arXiv:1511.05531).

\bibitem{Koblitz} N. Koblitz: ``Introduction to elliptic curves and modular forms,'' 2nd Ed.,  Graduate Texts in Mathematics \textbf{97}, Springer-Verlag (1993).

\bibitem{Ligo} G. Ligozat: ``Courbes modulaires de genre 1,'' Bull. Soc. Math. France \textbf{43}, Soc. Math. France, Paris (1975), 80 pp..

\bibitem{Lin} B. Lin: \emph{Elementary proofs of parity results for broken 3-diamond partitions},  J.  Number Theory \textbf{135} (2014),  1--7.

\bibitem{Mahl} K. Mahlburg: \emph{Partition congruences and the Andrews-Garvan-Dyson crank}, Proc. Natl. Acad. Sci. USA \textbf{102} (2005), no. 43, 15373--15376.

\bibitem{Monsky} P. Monsky: \emph{What does the computer suggest about the parity of $p(n)$, for $n$ in a fixed arithmetic progression?}, URL (as of October 2017): http://mathoverflow.net/q/36808.

\bibitem{Newman} M. Newman: \emph{Construction and application of a certain class of modular functions, II},  Proc. London Math. Soc. (3) \textbf{9} (1959), 353--387.

\bibitem{Newman2} M. Newman: \emph{Note on Partitions Modulo 5},  Math. Comp. \textbf{21} (1967), 481--482.

\bibitem{Nic} J.-L. Nicolas: \emph{Parit\'e des valeurs prises par la fonction de partition $p(n)$ et anatomie des entiers}, CRM Proceedings and Lecture Notes, Amer. Math. Soc. \textbf{46} (2008), 97--113.

\bibitem{NiSa} J.-L. Nicolas and A. S\'ark\"ozy: \emph{On the parity of partition functions}, Illinois J. Math. \textbf{39} (1995), no. 4, 586--597.

\bibitem{Ono1} K. Ono: \emph{Distribution of the partition function modulo $m$}, Ann. Math. \textbf{151} (2000), 293--307.

\bibitem{Ono} K. Ono: ``The Web of Modularity: Arithmetic of the Coefficients of Modular Forms and $q$-series,''  CBMS Regional Conference Series in Mathematics, Amer. Math. Soc. \textbf{102}, Providence, RI (2004).

\bibitem{Ono3} K. Ono: \emph{Parity of the partition function}, Adv. Math. \textbf{225} (2010), no. 1, 349--366.

\bibitem{PaSh} T.R. Parkin and D. Shanks: \emph{On the distribution of parity in the partition function}, Math. Comp. \textbf{21} (1967), 466--480.

\bibitem{Radu1} C.-S. Radu: \emph{A proof of Subbarao's conjecture}, J. Reine Angew. Math. \textbf{672} (2012), 161--175.

\bibitem{Radu} C.-S. Radu: \emph{An algorithmic approach to Ramanujan-Kolberg identities},  J. Symbolic Comput. \textbf{68} (2015), 225--253.

\bibitem{Rama2} S. Ramanujan: \emph{On Certain Arithmetical Functions}, Trans. Cambridge Philos. Soc. \textbf{22} (1916), no. 9, 159--184.

\bibitem{Rama} S. Ramanujan: \emph{Congruence properties of partitions}, Proc. London Math. Soc. (2) \textbf{19} (1919), 207--210.

\bibitem{Serre0} J.-P. Serre: \emph{Divisibilit\'{e}  des coefficients des formes modulaires de poids entier}, C.R. Acad. Sci. Paris A \textbf{279} (1974), 679--682.

\bibitem{sturm} J. Sturm: \emph{On the congruence of modular forms}, in: Lecture Notes in Math. \textbf{1240}, Springer (1984), 275--280.

\bibitem{Zan} F. Zanello:  \emph{On the number of odd values of the Klein $j$-function and the cubic partition function}, J. Number Theory \textbf{151} (2015), 107--115.

\bibitem{Zuckerman} J. Zuckerman:  \emph{Identities analogous to Ramanujan's identities involving the partition function},  Duke Math. J. \textbf{5} (1939), no. 1, 88--110.

\end{thebibliography}
\end{document}